\documentclass[final,3p]{elsarticle}
\usepackage[latin1]{inputenc}
 \usepackage{graphics}
 \usepackage{graphicx}
 \usepackage{epsfig}
\usepackage{amssymb}
 \usepackage{amsthm}
 \usepackage{lineno}
 \usepackage{amsmath}
\usepackage{color}
   \numberwithin{equation}{section}
\usepackage{mathrsfs}

\journal{ } 

\newtheorem{thm}{Theorem}[section]

\newtheorem{lem}[thm]{Lemma}

\newtheorem{defn}[thm]{Definition}
\newtheorem{rem}[thm]{Remark}
\newtheorem{ex}[thm]{Example}

\biboptions{sort&compress,square}
\allowdisplaybreaks
\begin{document}
\begin{frontmatter}
\author{Tong Wu$^{a}$}
\ead{wut977@nenu.edu.cn}
\author{Yong Wang$^{b,*}$}
\ead{wangy581@nenu.edu.cn}
\cortext[cor]{Corresponding author.}
\address{$^a$Department of Mathematics, Northeastern University, Shenyang, 110819, China}
\address{$^b$School of Mathematics and Statistics, Northeast Normal University,
Changchun, 130024, China}

\title{The Einstein-Hilbert action for perturbed second-order spectral triples}
\begin{abstract}
In \cite{FM}, the higher-order spectral triple and its relative K-homology were studied. Motivated by the Kastler-Kalau-Walze theorem, we propose an extension of the Einstein-Hilbert action to the framework of higher-order spectral triples. To illustrate this construction, we introduce two second-order spectral triples and explicitly compute their respective Einstein-Hilbert action, demonstrating the applicability of our theoretical framework.
\end{abstract}
\end{frontmatter}
\section{Introduction}
 Until now, many geometers have studied noncommutative residues. In \cite{Gu,Wo}, authors found noncommutative residues are of great importance to the study of noncommutative geometry. Connes showed us that the noncommutative residue on a compact manifold $M$ coincided with the Dixmier's trace on pseudodifferential operators of order $-{\rm {dim}}M$ in \cite{Co2}. Therefore, the noncommutative residue can be used as integral of noncommutative geometry and become an important tool of noncommutative geometry. In \cite{Co1}, Connes used the noncommutative residue to derive a conformal 4-dimensional Polyakov action analogy.
Several years ago, Connes made a challenging observation that the noncommutative residue of the square of the inverse of the Dirac operator $D$ was proportional to the Einstein-Hilbert action, which we call the Kastler-Kalau-Walze theorem. In \cite{Ka}, Kastler gave a bruteforce proof of this theorem. In \cite{KW}, Kalau and Walze proved this theorem in the normal coordinates system simultaneously. The Kastler-Kalau-Walze theorem gives a spectral explanation of the gravitational action, it says that there exists a constant $c_0$, such that
 $${\rm Wres (D^{-2})}=c_0\int_M sd{\rm Vol_M},$$
where ${\rm Wres}$ denotes the noncommutative residue and $s$ is the scalar curvature. Ackermann proved that the Wodzicki residue of the square of the inverse of the Dirac operator ${\rm Wres}(D^{-2})$ in turn is essentially the second coefficient of the heat kernel expansion of $D^{2}$ in \cite{Ac}.

 In Connes' program of noncommutative geometry, the role of geometrical objects is played by spectral triples $(\mathcal{A},\mathcal{H},D)$.
Similar to the commutative case and the canonical spectral triple $(C^\infty(M),L^2(S),D)$, where $(M,g,S)$ is a closed spin
manifold and $D$ is the Dirac operator acting on the spinor bundle $S$, the spectrum of the Dirac operator $D$ of a spectral
triple $(\mathcal{A},\mathcal{H},D)$ encodes the geometrical information of spectral triple.
Many spectral triples on spin manifolds have been studied. The trilinear functional of differential one-forms for a finitely summable regular spectral triple with a noncommutative
residue has been computed recently in \cite{D2}, Dabrowski et al. demonstrated that for a canonical spectral triple over a closed
spin manifold it recovers the torsion of the linear connection, which is a first step towards linking the spectral approach
with the algebraic approach based on Levi-Civita connections. In \cite{Si1}, Sitarz proposed a new idea of conformally rescaled and curved spectral triples, which
are obtained from a real spectral triple by a nontrivial scaling of the Dirac
operator. And they computed the Wodzicki residue and the Einstein-Hilbert functional for such family on the 4-dimensional noncommutative torus. Sitarz computed the Wodzicki residue of the inverse of a conformally rescaled
Laplace operator over a 4-dimensional noncommutative torus in \cite{Si2}. In \cite{FM}, Fries M extended the notion of a spectral triple to that of a higher-order relative spectral triple, and calculate the
K-homology boundary map of the constructed relative K-homology cycle in the case of an elliptic differential operator on a compact smooth manifold with boundary to obtain a generalization of the Baum-Douglas-Taylor index theorem. Therefore, we intend to generalize the Einstein-Hilbert action to higher-order spectral triples, that is the {\bf motivation} of this article. Based on by the relationship between the Einstein-Hilbert action of general relativity and first-order spectral triples, we construct two kinds of perturbed second-order spectral triples with the noncommutative residue to explore the Einstein-Hilbert action under perturbed second-order spectral triples. The aim of this paper is to prove the following theorems.
\begin{thm}\label{1thmooo}
Let $M$ be a $2m$-dimensional oriented
compact spin manifold without boundary, then the Einstein-Hilbert action for perturbed second-order spectral triples-type I is given
 \begin{align*}
{\rm Wres}[(fD^2f)^{-m+1}]
&=2^m\int_M\bigg(-\frac{1}{12}(m-1)f^{-2m+2}s+\frac{1}{3}(m^2-3m+2)f^{-2m+1}\Delta(f)\nonumber\\
&+\frac{1}{3}f^{-2m}m(m^2-3m+2)|\nabla(f)|^2\bigg)Vol(S^{n-1})d{\rm Vol_M},
 \end{align*}
 where $\Delta(f)$ and $\nabla(f)$ denote a generalized laplacian of $f$ and the gradient to $f$ respectively.
 \end{thm}
 \begin{rem}
 (1)When $f=1,$ we get the classicial Kastler-Kalau-Walze theorem.
 \begin{align*}
{\rm Wres}[D^{-2}]
&=2^m\int_M-\frac{m-1}{12}sVol(S^{n-1})d{\rm Vol_M}.
 \end{align*}

(2)When $m=2,$ we get the Einstein-Hilbert action on 4-dimensional manifolds.
 \begin{align*}
{\rm Wres}[(fD^2f)^{-1}]
&=\int_M-\frac{1}{3}f^{-2}sVol(S^{n-1})d{\rm Vol_M}.
 \end{align*}
 \end{rem}

\begin{thm}\label{666thmooo}
Let $M$ be a $2m$-dimensional oriented
compact spin manifold without boundary, then the Einstein-Hilbert action for perturbed second-order spectral triples-type II is given
 \begin{align*}
{\rm Wres}[(c(X)D^2c(X))^{-m+1}]
&=(-1)^m2^m\int_M\bigg(\frac{1}{12}(m-1)|X|^{-2m+2}s+\frac{1}{6}(m^2-3m+2)|X|^{-2m+4}\Delta(|X|^{-2})\nonumber\\
&+\frac{1}{12m}(6m^3-m^4+m^2-30m+24)|X|^{-2m+6}|\nabla(|X|^{-2})|^2\bigg)Vol(S^{n-1})d{\rm Vol_M},
 \end{align*}
 where $\Delta(|X|^{-2})$ and $\nabla(|X|^{-2})$ denote a generalized laplacian of $|X|^{-2}$ and the gradient to $|X|^{-2}$ respectively.

 \end{thm}
\indent The paper is organized in the following way. In Section \ref{section:2}, firstly, we introduce higher-order spectral triples. On this basis, we construct perturbed second-order spectral triples. In Section \ref{section:3}, we define the associated Einstein-Hilbert action and combine the noncommutative residue formula, explicitly compute the Einstein-Hilbert action for two perturbed second-order spectral triples.
\section{Constructing perturbed second-order spectral triples}
\label{section:2}
 In this section, we present examples of perturbed second-order spectral triples. Firstly, we recall the definition of the higher-order spectral triple.
\begin{defn}\label{de1}\cite{FM}
 Let $A$ be a $C^*$-algebra, define a higher-order spectral triple $(\mathcal{A},\mathcal{H},\mathcal{D})$ of order $k>0$ for $a$ as consisting a $\mathbb{Z}/2$-graded separable
Hilbert space $\mathcal{H}$, an even representation $\rho: A\rightarrow \mathbb{B}(\mathcal{H})$, a dense sub-$\ast$-algebras $\mathcal{A}\subseteq A$
 and an odd self-adjoint unbounded operator $\mathcal{D}$ on $\mathcal{H}$ such that for each $a\in A$\\
 1. $\rho(a)Dom \mathcal{D}\subseteq Dom \mathcal{D}$;\\
 2. $\rho(a)(1+\mathcal{D}^2)^{-\frac{1}{2}}\in \mathbb{K}(\mathcal{H})$;\\
 3. $[\mathcal{D},\rho(a)](1+\mathcal{D}^2)^{-\frac{1}{2}+\frac{1}{2k}}$ extends to an element in $\mathbb{B}(\mathcal{H})$.
\end{defn}
Moreover, let $M$ be a closed manifold, we give two examples about higher-order spectral triples.
\begin{ex}
Let $T$ be a formally self-adjoint elliptic differential operator, then $(C^\infty(M),L^2(M;E),T)$ is a higher-order spectral triples, where $M$ is compact manifold without boundary.
 \end{ex}
 \begin{ex}
Let $T$ be a formally self-adjoint elliptic differential operator, then  $(C^\infty(M)\otimes M_N(\mathbb{C}),T\otimes Id_N,L^2(M;E)\otimes M_N(\mathbb{C}))$ is a almost commutative higher-order spectral triples.
 \end{ex}
 \begin{ex}
 Let $\lambda_1,\cdot\cdot\cdot,\lambda_n$ be real number, $D$ is the Dirac operator, then $(C^\infty(M),\mathcal{H},D^n+\lambda_1D^{n-1}+\cdot\cdot\cdot+\lambda_{n-1}D+\lambda_n)$ is a higher-order spectral triples.
 \end{ex}

For constructing perturbed second order spectral triples, we introduce the Dirac operator $D$.
Let $M$ be an $n=2m$-dimensional oriented compact closed Riemannian spin manifold with a Riemannian metric $g$ and let $\nabla^L$ be the Levi-Civita connection about $g$. Then in the local coordinates $\{x_i; 1\leq i\leq n\}$ and the
fixed orthonormal frame $\{e_1,\cdots,e_n\}$, the connection matrix $(\omega_{s,t})$ is defined by
\begin{equation}
\nabla^L(e_1,\cdots,e_n)= (e_1,\cdots,e_n)(\omega_{s,t}).
\end{equation}
\indent Suppose that $\partial_{i}$ is a natural local frame on $TM$
and $(g^{ij})_{1\leq i,j\leq n}$ is the inverse matrix associated to the metric
matrix  $(g_{ij})_{1\leq i,j\leq n}$ on $M$. Write
\begin{equation}
\widehat{c}(e_j)=\epsilon (e_j* )+\iota
(e_j*);~~
c(e_j)=\epsilon (e_j* )-\iota (e_j* ).
\end{equation}
which satisfies
\begin{align}\label{kaka}
&\widehat{c}(e_i)\widehat{c}(e_j)+\widehat{c}(e_j)\widehat{c}(e_i)=2\delta_{i}^j;~~\nonumber\\
&c(e_i)c(e_j)+c(e_j)c(e_i)=-2\delta_{i}^j;~~\nonumber\\
&c(e_i)\widehat{c}(e_j)+\widehat{c}(e_j)c(e_i)=0,
\end{align}
where $\epsilon (e_j*)$,~$\iota (e_j*)$ are the exterior and interior multiplications respectively and $c(e_j)$ is the Clifford action.
Then we have
\begin{align}
D&=\sum^n_{i=1}c(e_i)\bigg[e_i-\frac{1}{4}\sum_{s,t}\omega_{s,t}
(e_i)c(e_s)c(e_t)\bigg].
\end{align}
Using $g^{ij}=g(dx_{i},dx_{j})$, $\xi=\sum_{j}\xi_{j}dx_{j}$ and $\nabla^L_{\partial_{i}}\partial_{j}=\sum_{k}\Gamma_{ij}^{k}\partial_{k}$,  we denote that
\begin{align}\label{54}
&\sigma_{i}=-\frac{1}{4}\sum_{s,t}\omega_{s,t}
(e_i)c(e_s)c(e_t)
;~~~\xi^{j}=g^{ij}\xi_{i};~~~\Gamma^{k}=g^{ij}\Gamma_{ij}^{k};~~~\sigma^{j}=g^{ij}\sigma_{i};~~~\partial^j=g^{ij}\partial_i.
\end{align}
Consequently, the Dirac operator $D$ can be written as
\begin{equation}
D=\sum^n_{i=1}c(e_i)[e_i+\sigma_{i}].
\end{equation}
By (6a) in \cite{Ka}, we get
\begin{align}\label{lal}
{D}^2&=-\sum_{i,j}g^{ij}\Big[\partial_{i}\partial_{j}+2\sigma_{i}\partial_{j}-\Gamma_{ij}^{k}\partial_{k}+(\partial_{i}\sigma_{j})
+\sigma_{i}\sigma_{j}-\Gamma_{ij}^{k}\sigma_{k}\Big]
+\frac{1}{4}s,
\end{align}
where $s$ is the scalar curvature.\\
Then by (\ref{54}), it follows that respective symbols of $D^{2}$.
\begin{lem}\cite{Ka}\label{lemma1}Some symbols of $D^{2}$ are given:
\begin{align}
\sigma_2(D^2)&=|\xi|^2;\nonumber\\
\sigma_1(D^2)&=\sqrt{-1}(\Gamma^\mu-2\sigma^\mu)(x)\xi_\mu;\nonumber\\
\sigma_0(D^2)&=-(\partial^{\mu}\sigma_\mu+\sigma^\mu\sigma_\nu-\Gamma^{\mu}\sigma_\mu)(x)+\frac{1}{4}s(x).
\end{align}
\end{lem}
 Based on the Dirac operator, it is natural to construct two kinds of perturbed second-order spectral triples.
 \begin{ex}\label{e1}
Let $f$ be a smooth function on $M$ and $f(x)\neq 0$ for any $x\in M$, then $(C^\infty(M)\otimes M_N(\mathbb{C}), fD^2f, L^2(M,S(TM)\otimes M_N(\mathbb{C})))$ is a almost commutative second-order spectral triple. When $N=1,$ we can get the second-order spectral triple $(C^\infty(M), fD^2f, L^2(M,S(TM)))$.
 \end{ex}
 \begin{ex}\label{e2}
Let $X$ be a vector field and $X(p)\neq 0$ for any $p\in M$, and let $c(X)$ denote Clifford action on manifold $M$, then $(C^\infty(M)\otimes M_N(\mathbb{C}), c(X)D^2c(X), L^2(M,S(TM)\otimes M_N(\mathbb{C})))$ is a almost commutative second-order spectral triples.
 When $N=1,$ we can get the second-order spectral triple $(C^\infty(M), c(X)D^2c(X),\\
  L^2(M,S(TM))).$
 \end{ex}

 \section{ The computations of the Einstein-Hilbert action for perturbed second-order spectral triples}
\label{section:3}
 In this section we want to consider the Einstein-Hilbert action for two perturbed second-order spectral triples.

For the $2m-$dimensional spin manifold $M$ and the Dirac operator $D$, we have the commutative higher-order spectral triple by $(C^\infty(M), D^l, L^2(M,S(TM)))$, where $2l\geq 2$ and $\frac{2-2m}{l}$ are both an integer. By the Kastler-Kalau-Walze theorem, we obtain
\begin{align}
{\rm Wres}(D^l)^{\frac{2-2m}{l}}:=c_0\int_Msd{\rm Vol_M}.
\end{align}
This naturally leads us to the following key definition.
\begin{defn}For the general $l-$order spectral triples $(\mathcal{A},\mathcal{H},L),$ define the associated Einstein-Hilbert action is ${\rm Wres}(L^{\frac{2-2m}{l}}),$ where $\frac{2-2m}{l}$ is an integer and ${\rm Wres}$ is a trucial state on the algebra generated by $a\in \mathcal{A}$, $[L,a]$, $L^r$, $r\in Z.$
\end{defn}
Since considering the second-order spectral triples, take $l=2,$ it follows that the following noncommutative residue formula \cite{Wo1,Ka} plays a key role in our computations:
\begin{align}\label{666}
{\rm Wres}[\mathcal{D}^{-m+1}]:=\int_M\int_{|\xi|=1}{\rm tr}[\sigma_{-2m}({\mathcal{D}^{-m+1}})](x,\xi)\sigma(\xi)dx,
\end{align}
where $\sigma_{-2m}({\mathcal{D}^{-m+1}})$ denotes the ($-2m$)th order piece of the complete symbols of ${\mathcal{D}^{-m+1}}$, {\rm tr} as shorthand of trace.

In the following computations, we choose two perturbed second-order spectral triples in Example \ref{e1} and \ref{e2}, that is, take $\mathcal{D}=fD^2f$ and $\mathcal{D}=c(X)D^2c(X)$. Overall computations are based on the algorithm yielding the principal symbol of a
product of pseudo-differential operators in terms of the principal symbols of the
factors, namely,
\begin{align}\label{1111}
\sigma^{AB}(x,\xi)=\sum_\alpha\frac{(-\sqrt{-1})^\alpha}{\alpha!}\partial_{\xi_\alpha}\sigma^{A}(x,\xi)\cdot\partial_{x_\alpha}\sigma^{B}(x,\xi).
\end{align}
 Write $
~\sigma(A^{-1})=\sum^{\infty}_{j=2}b_{-j},
$
 then by the composition formula of pseudodifferential operators, we have
\begin{align}
1&=\sigma(AA^{-1})(x,\xi)\nonumber\\
&=\sum_\alpha\frac{(-\sqrt{-1})^\alpha}{\alpha!}\partial_{\xi_\alpha}\sigma(A)(x,\xi)\cdot\partial_{x_\alpha}\sigma(A^{-1})(x,\xi)\nonumber\\
&=(\sigma_2+\sigma_1+\sigma_0)(b_{-2}+b_{-3}+b_{-4}+\cdots)\nonumber\\
&~~~+\sum_\alpha(\partial_{\xi_\alpha}\sigma_2+\partial_{\xi_\alpha}\sigma_1+\partial_{\xi_\alpha}\sigma_0)(
-\sqrt{-1}\partial_{x_\alpha}q_{-2}-\sqrt{-1}\partial_{x_\alpha}b_{-3}-\sqrt{-1}\partial_{x_\alpha}b_{-4}+\cdots).
\end{align}
This leads us to the following.
\begin{align}\label{g3}
&b_{-2}=\sigma_2^{-1};\nonumber\\
&b_{-3}=-\sigma_2^{-1}[\sigma_1b_{-2}-\sqrt{-1}\partial_{\xi_a}(\sigma_2)\partial_{x_a}(b_{-2})];\nonumber\\
&b_{-4}=-\sigma_2^{-1}[\sigma_1b_{-3}+\sigma_0b_{-2}-\sqrt{-1}\partial_{\xi_a}(\sigma_1)\partial_{x_a}(b_{-2})-\sqrt{-1}\partial_{\xi_a}(\sigma_2)\partial_{x_a}(b_{-3})-\frac{1}{2}\partial_{\xi_a}\partial_{\xi_b}(\sigma_2)\partial_{x_a}\partial_{x_b}(b_{-2})].
\end{align}

Based on this, in order to calculate $\sigma_{-2m}({\mathcal{D}^{-m+1}})$, we also need the following important conclusion. It follows from Corollary A.2 in \cite{D1} that the following holds when $k=2,~l=m-1$.
\begin{align}\label{g1}
\sigma(\mathcal{D}^{-m+1})_{-2m}&=(m-1)(b_{-2})^{m-2}b_{-4}+\frac{(m-1)(m-2)}{2}(b_{-2})^{m-3}(b_{-3})^2-\sqrt{-1}\frac{(m-1)(m-2)}{2}\nonumber\\
&(b_{-2})^{m-4}\bigg[b_{-2}\partial_{\xi_a}(b_{-3})\partial_{x_a}(b_{-2})+\partial_{\xi_a}(b_{-2})\partial_{x_a}(b_{-3})+(m-3)b_{-3}\partial_{\xi_a}(b_{-2})\partial_{x_a}(b_{-2})\bigg]\nonumber\\
&-\frac{(m-1)(m-2)}{24}(b_{-2})^{m-5}\bigg[6(b_{-2})^2\partial_{\xi_a}\partial_{\xi_b}(b_{-2})\partial_{x_a}\partial_{x_b}(b_{-2})+3(m-3)(m-4)\nonumber\\
&\partial_{\xi_a}(b_{-2})\partial_{\xi_b}(b_{-2})\partial_{x_a}(b_{-2})\partial_{x_b}(b_{-2})+4(m-3)b_{-2}\bigg(\partial_{\xi_a}(b_{-2})\partial_{\xi_b}(b_{-2})\partial_{x_a}\partial_{x_b}(b_{-2})\nonumber\\
&+\partial_{\xi_a}(b_{-2})\partial_{\xi_b}\partial_{x_a}(b_{-2})\partial_{x_b}(b_{-2})+\partial_{\xi_a}\partial_{\xi_b}(b_{-2})\partial_{x_a}(b_{-2})\partial_{x_b}(b_{-2})\bigg)\bigg].
\end{align}

 \subsection{ The perturbed second-order spectral triples-tpye I}

For perturbed second-order spectral triples-tpye I $\mathcal{D}=fD^2f,$ we have
$$fD^2f=f^2D^2+f[D^2,f].$$
The computation of symbolys of $[D^2,f]$ need to use this lemma below.
\begin{lem}\cite{UW}\label{lema2}
Let $S$ be a pseudo-differential operator of order $k$ and $f$ is a smooth function, $[S,f]$ is a pseudo-differential operator of order $k-1$ with total symbol
$\sigma[S,f]\sim\sum_{j\geq 1}\sigma_{k-j}[S,f]$, where
\begin{align}
\sigma_{k-j}[S,f]=\sum_{|\beta|=1}^j\frac{D_x^\beta(f)}{\beta!}\partial_{\xi_\beta}(\sigma^S_{k-(j-|\beta|)}).
\end{align}
\end{lem}
 Now we consider the explicit representation of the symbols of $\sigma_{-2m}[(fD^2f)^{-m+1}]$, decompose the operators $fD^2f$ by different orders as
\begin{align}
\sigma(fD^2f)=\sigma_2(fD^2f)+\sigma_1(fD^2f)+\sigma_0(fD^2f).
\end{align}
By Lemma \ref{lemma1} and Lemma \ref{lema2}, we have the following lemma.
\begin{lem}\label{lemma4}Some symbols of positive order for the $fD^2f$ are given
\begin{align}\label{g2}
\sigma_2(fD^2f)&=f^2|\xi|^2;\nonumber\\
\sigma_1(fD^2f)&=\sqrt{-1}f^2(\Gamma^\mu-2\sigma^\mu)\xi_\mu-2\sqrt{-1}f\partial_{x_j}(f)\xi^j;\nonumber\\
\sigma_0(fD^2f)&=-f^2(\partial^{\mu}\sigma_\mu+\sigma^\mu\sigma_\nu-\Gamma^{\mu}\sigma_\mu)+\frac{1}{4}f^2s+f \partial_{x_j}(f)(\Gamma^j-2\sigma^j)-f\partial_{x_j}\partial_{x_l}(f)g^{jl}.
\end{align}
\end{lem}
For any fixed point $x_0\in M$, choosing the normal coordinates $U$ of $x_0$ in $M$, then the following results hold.
\begin{align}\label{mmnb}
\sigma^j(x_0)=0;~~~\Gamma^k(x_0)=0;~~~\partial_{x_a}g^{\alpha\beta}(x_0)=0;~~~g^{ij}(x_0)=\delta_i^j.
\end{align}
Besides, the following results from \cite{Ka} is very important for our following calculations.
\begin{align}\label{mkmk}
&\partial_{\xi_\mu}(|\xi|^2)=2\xi^\mu;~~~~~\partial_{\xi_\mu}(|\xi|^{-2})=-2|\xi|^{-4}\xi^\mu;\nonumber\\
&\partial_{\xi_\mu}(|\xi|^{-4})=-4|\xi|^{-6}\xi^\mu;~~~~\partial_{\xi_\mu}(|\xi|^{-6})=-6|\xi|^{-8}\xi^\mu;\nonumber\\
&\partial_{x_\mu}(|\xi|^{2})=\xi^\alpha\xi^\beta\partial_{x_\mu}g^{\alpha\beta};~~~\partial_{x_\mu}(|\xi|^{-2})=-|\xi|^{-4}\xi_\alpha\xi_\beta\partial_{x_\mu}g^{\alpha\beta};\nonumber\\
&\partial_{x_\mu}(|\xi|^{-6})=-3|\xi|^{-8}\xi_\alpha\xi_\beta\partial_{x_\mu}g^{\alpha\beta};~~~\partial_{\xi_\mu}\partial_{\xi_\nu}(|\xi|^2)=2g^{\mu\nu};\nonumber\\
&\partial_{x_\mu}\partial_{x_\nu}(|\xi|^{-2})=-|\xi|^{-4}\xi_\alpha\xi_\beta\partial_{x_\mu}\partial_{x_\nu}g^{\alpha\beta}+2|\xi|^{-6}\xi_\alpha\xi_\beta\partial_{x_\mu}g^{\alpha\beta}\xi_\gamma\xi_\delta\partial_{x_\nu}g^{\gamma\delta}.
\end{align}
By (\ref{g3}), (\ref{g2})-(\ref{mkmk}), we obtain that
\begin{lem}\label{lemma34}Some symbols of negative order for the $(fD^2f)^{-1}$ are given
\begin{align}\label{g23}
b_{-2}[(fD^2f)^{-1}]&=f^{-2}|\xi|^{-2};\nonumber\\
b_{-3}[(fD^2f)^{-1}]&=-\sqrt{-1}f^{-2}|\xi|^{-4}(\Gamma^\mu-2\sigma^\mu)\xi_\mu-2\sqrt{-1}f^{-3}|\xi|^{-4}\partial_{x_j}(f)\xi^j\nonumber\\
&-2\sqrt{-1}f^{-2}|\xi|^{-6}\xi^\mu\xi_\alpha\xi_\beta\partial_{x_\mu}g^{\alpha\beta};\nonumber\\
b_{-4}[(fD^2f)^{-1}](x_0)&=-\frac{1}{4}f^{-2}|\xi|^{-4}s-f^{-3}|\xi|^{-4}\partial_{x_j}^2(f)+2f^{-4}|\xi|^{-4}(\partial_{x_j}(f))^2\nonumber\\
&+\frac{2}{3}f^{-2}|\xi|^{-6}R_{\alpha a\alpha\mu}(x_0)\xi_\mu\xi_a-8f^{-4}|\xi|^{-6}\partial_{x_j}(f)\partial_{x_a}(f)\xi_j\xi_a\nonumber\\
&+4f^{-3}|\xi|^{-6}\partial_{x_a}\partial_{x_j}(f)\xi_j\xi_a.
\end{align}
\end{lem}
Substituting (\ref{g2}) and (\ref{g23}) into (\ref{g1}), the following result can be obtained.
\begin{align}\label{g13}
\sigma[(fD^2f)^{-m+1}]_{-2m}(x_0)&=-\frac{m-1}{4}f^{-2m+2}|\xi|^{-2m}s\nonumber\\
&-(m-1)^2f^{-2m+1}|\xi|^{-2m}\partial_{x_j}^2(f)\nonumber\\
&+\frac{1}{3}m(m-1)f^{-2m+2}|\xi|^{-2m-2}R_{\alpha a\alpha\mu}(x_0)\xi_\mu\xi_a\nonumber\\
&+\frac{1}{3}m(4m^2-9m+5)f^{-2m}|\xi|^{-2m}(\partial_{x_j}(f))^2\nonumber\\
&+\frac{2}{3}m(2m^2-3m+1)f^{-2m+1}|\xi|^{-2m-2}\partial_{x_j}\partial_{x_l}(f)\xi_j\xi_l\nonumber\\
&-2m^2(m-1)^2f^{-2m}|\xi|^{-2m-2}\partial_{x_a}(f)\partial_{x_b}(f)\xi_a\xi_b.
\end{align}

Next, we need to take the trace of the above result and integrate them. In which a relatively important integral formula is used as follows.
\begin{align}\label{wer}
&\int_{|\xi|=1}\xi_j\xi_l \sigma(\xi)=\frac{1}{2m}\delta_{jl}Vol(S^{n-1}).
\end{align}
The results of the integration with respect to each of (\ref{g13}) are as follows.
\begin{lem}\label{ppp} The following identities hold:
\begin{align}\label{ppo}
&\int_{|\xi|=1}{\rm tr}\bigg(f^{-2m+2}|\xi|^{-2m}s\bigg)\sigma(\xi)=f^{-2m+2}s{\rm tr}[id]Vol(S^{n-1});\nonumber\\
&\int_{|\xi|=1}{\rm tr}\bigg(f^{-2m+2}|\xi|^{-2m-2}R_{\alpha a\alpha \mu}(x_0)\xi_a\xi_\mu\bigg)\sigma(\xi)=\frac{1}{2m}f^{-2m+2}s{\rm tr}[id]Vol(S^{n-1});\nonumber\\
&\int_{|\xi|=1}{\rm tr}\bigg(f^{-2m+1}|\xi|^{-2m}\sum_{j}\partial_{x_j}^2(f)\bigg)\sigma(\xi)=-f^{-2m+1}\Delta(f){\rm tr}[id]Vol(S^{n-1});\nonumber\\
&\int_{|\xi|=1}{\rm tr}\bigg(f^{-2m}|\xi|^{-2m}\sum_{j}(\partial_{x_j}(f))^2\bigg)\sigma(\xi)=f^{-2m}|\nabla(f)|^2{\rm tr}[id]Vol(S^{n-1});\nonumber\\
&\int_{|\xi|=1}{\rm tr}\bigg(f^{-2m+1}|\xi|^{-2m-2}\sum_{j,l}\partial_{x_j}\partial_{x_l}(f)\xi_j\xi_l\bigg)\sigma(\xi)=-\frac{1}{2m}f^{-2m+1}\Delta(f){\rm tr}[id]Vol(S^{n-1});\nonumber\\
&\int_{|\xi|=1}{\rm tr}\bigg(f^{-2m}|\xi|^{-2m-2}\sum_{j,l}\partial_{x_j}(f)\partial_{x_l}(f)\xi_j\xi_l\bigg)\sigma(\xi)=\frac{1}{2m}f^{-2m}|\nabla(f)|^2{\rm tr}[id]Vol(S^{n-1}),
\end{align}
where $\Delta(f)$ and $\nabla(f)$ denote a generalized laplacian of $f$ and the gradient to $f$ respectively.
\end{lem}
By Lemma \ref{ppp}, integrate each term in (\ref{g1}) and then substitute into (\ref{666}) to obtain the following theorem.
\begin{thm}\label{thmooo}
Let $M$ be a $2m$-dimensional oriented
compact spin manifold without boundary, then the Einstein-Hilbert action for perturbed second-order spectral triples-type I is given
 \begin{align}
{\rm Wres}[(fD^2f)^{-m+1}]
&=2^m\int_M\bigg(-\frac{1}{12}(m-1)f^{-2m+2}s+\frac{1}{3}(m^2-3m+2)f^{-2m+1}\Delta(f)\nonumber\\
&+\frac{1}{3}f^{-2m}m(m^2-3m+2)|\nabla(f)|^2\bigg)Vol(S^{n-1})d{\rm Vol_M}.
 \end{align}
 \end{thm}

 \subsection{ The perturbed second-order spectral triples-tpye II}
For perturbed second-order spectral triples-tpye II $\mathcal{D}=c(X)D^2c(X),$ we need to consider the explicit representation of the symbols of $\sigma_{-2m}[(c(X)D^2c(X))^{-m+1}]$. Decompose the operators $c(X)D^2c(X)$ by different orders as
\begin{align}
\sigma[c(X)D^2c(X)]=\sigma_2[c(X)D^2c(X)]+\sigma_1[c(X)D^2c(X)]+\sigma_0[c(X)D^2c(X)].
\end{align}
To obtain the different orders of symbol of $c(X)D^2c(X)$, we act $c(X)$ on (\ref{lal}). According to the Leibniz rule,  we can obtain
$$\partial_j\circ c(X)=c(X)\partial_j+\partial_j[c(X)].$$
Moreover, the following can be obtained.
\begin{align}
\partial_i \partial_j \circ c(X)=c(X)\partial_i \partial_j+\partial_i[c(X)]\partial_j+\partial_j[c(X)]\partial_i+\partial_i \partial_j[c(X)].
\end{align}
Consequently, we obtain that
\begin{align}
c(X)D^2c(X)&=g^{ij}\bigg(|X|^2\partial_i\partial_j-c(X)\partial_i[c(X)]\partial_j-c(X)\partial_j[c(X)]\partial_i+2|X|^2\sigma_i\partial_j-|X|^2\Gamma^k_{ij}\partial_k-c(X)\partial_i \partial_j[c(X)]\nonumber\\
&-2c(X)\sigma_i \partial_j[c(X)]+c(X)\Gamma^k_{ij} \partial_k[c(X)]+|X|^2\sigma_i\sigma_j-|X|^2\Gamma^k_{ij}\sigma_k+|X|^2(\partial_i\sigma_j)\bigg)-\frac{1}{4}|X|^2s.
\end{align}
Using $\sigma(\partial_j)=\sqrt{-1}\xi_j$, we get the lemma as follows.
\begin{lem}\label{glemma4}Some symbols of positive order for the $c(X)D^2c(X)$ are given
\begin{align}\label{g4}
\sigma_2[c(X)D^2c(X)]&=-|X|^2|\xi|^2;\nonumber\\
\sigma_1[c(X)D^2c(X)]&=-\sqrt{-1}|X|^2(\Gamma^\mu-2\sigma^\mu)\xi_\mu-\sqrt{-1}c(X)\partial^i[c(X)]\xi_j-\sqrt{-1}c(X)\partial^j[c(X)]\xi_i;\nonumber\\
\sigma_0[c(X)D^2c(X)]&=|X|^2(\partial_j\sigma_j+\sigma^j\sigma_j-\Gamma^k\sigma_k)-\frac{1}{4}|X|^2s-c(X)\partial^j\partial_j[c(X)]-2c(X)\sigma^j\partial_j[c(X)]\nonumber\\
&+c(X)\Gamma^k\partial_k[c(X)].
\end{align}
\end{lem}
By (\ref{g3}), (\ref{mmnb}), (\ref{mkmk}) and (\ref{g4}), we obtain that, we have
\begin{lem}\label{glemma5}Some symbols of negative order for the $(c(X)D^2c(X))^{-1}$ are given
\begin{align}\label{g5}
b_{-2}[(c(X)D^2c(X))^{-1}]&=-|X|^{-2}|\xi|^{-2};\nonumber\\
b_{-3}[(c(X)D^2c(X))^{-1}]&=\sqrt{-1}|X|^{-2}|\xi|^{-4}(\Gamma^k-2\sigma^k)\xi_k -2\sqrt{-1}|\xi|^{-4}\partial_{x_j}(|X|^{-2})\xi^j\nonumber\\
&+2\sqrt{-1}|X|^{-2}|\xi|^{-6}\xi^\mu\xi_\alpha\xi_\beta\partial_{x_\mu}g^{\alpha\beta}+2\sqrt{-1}|X|^{-4}|\xi|^{-4}c(X)\partial^j[c(X)]\xi_j;\nonumber\\
b_{-4}[(c(X)D^2c(X))^{-1}](x_0)&=\frac{1}{4}|X|^{-2}|\xi|^{-4}s-\frac{2}{3}|X|^{-2}|\xi|^{-6}R_{\alpha a\alpha\mu}(x_0)\xi_\mu\xi_a+4|\xi|^{-6}\partial_{x_a}\partial_{x_\mu}(|X|^{-2})\xi_a\xi_\mu\nonumber\\
&-|\xi|^{-4}\partial_{x_a}^2(|X|^{-2})+4|X|^{-6}|\xi|^{-6}c(X)\partial_{x_a}[c(X)]c(X)\partial_{x_j}[c(X)]\xi_a\xi_j\nonumber\\
&-4|X|^{-4}|\xi|^{-6}\partial_{x_a}[c(X)]\partial_{x_j}[c(X)]\xi_a\xi_j+2|X|^{-2}|\xi|^{-4}\partial_{x_j}(|X|^{-2})c(X)\partial_{x_j}[c(X)]\nonumber\\
&+|X|^{-4}|\xi|^{-4}c(X)\partial_{x_j}\partial_{x_j}[c(X)]-12|X|^{-2}|\xi|^{-6}\partial_{x_j}(|X|^{-2})c(X)\partial_{x_a}[c(X)]\xi_a\xi_j\nonumber\\
&-4|X|^{-4}|\xi|^{-6}c(X)\partial_{x_a}\partial_{x_j}[c(X)]\xi_a\xi_j.
\end{align}
\end{lem}
Substituting (\ref{g4}) and (\ref{g5}) into (\ref{g1}), the following result can be obtained.
\begin{align}\label{mnmn}
&\sigma[(c(X)D^2c(X))^{-m+1}]_{-2m}(x_0)\nonumber\\
&=\frac{1}{4}(m-1)(-1)^m|X|^{-2m+2}|\xi|^{-2m}s\nonumber\\
&-\frac{1}{3}(m^2-m)(-1)^m|X|^{-2m+2}|\xi|^{-2m-2}R_{\alpha a\alpha\mu}(x_0)\xi_\mu\xi_a\nonumber\\
&-\frac{1}{2}(m^2-m)(-1)^m|X|^{-2m+4}|\xi|^{-2m}\partial_{x_a}^2(|X|^{-2})\nonumber\\
&+\frac{2}{3}m(m^2-1)(-1)^m|X|^{-2m+4}|\xi|^{-2m-2}\partial_{x_j}\partial_{x_a}(|X|^{-2})\xi_j\xi_a\nonumber\\
&+\frac{1}{2}m(m^3-2m^2-m+2)(-1)^m|X|^{-2m+6}|\xi|^{-2m-2}\partial_{x_\mu}(|X|^{-2})\partial_{x_a}(|X|^{-2})
\xi_\mu\xi_a\nonumber\\
&-\frac{1}{3}m(m^2-3m+2)(-1)^m|X|^{-2m+6}|\xi|^{-2m}(\partial_{x_j}(|X|^{-2}))^2\nonumber\\
&-2m(m-1)(-1)^m|X|^{-2m}|\xi|^{-2m-2}\partial_{x_a}[c(X)]\partial_{x_j}[c(X)]\xi_a\xi_j\nonumber\\
&-2(m^3-2m^2+5m-4)(-1)^m|X|^{-2m+2}|\xi|^{-2m-2}\partial_{x_j}(|X|^{-2})c(X)\partial_{x_a}[c(X)]\xi_a\xi_j\nonumber\\
&+(m-1)(-1)^m|X|^{-2m}|\xi|^{-2m}c(X)\partial_{x_j}\partial_{x_j}[c(X)]\nonumber\\
&+m(m-1)(-1)^m|X|^{-2m+2}|\xi|^{-2m}\partial_{x_j}(|X|^{-2})c(X)\partial_{x_j}[c(X)]\nonumber\\
&-2m(m-1)(-1)^m|X|^{-2m}|\xi|^{-2m-2}c(X)\partial_{x_a}\partial_{x_j}[c(X)]\xi_a\xi_j\nonumber\\
&+2m(m-1)(-1)^m|X|^{-2m-2}|\xi|^{-2m-2}c(X)\partial_{x_j}[c(X)]c(X)\partial_{x_a}[c(X)]\xi_a\xi_j.
\end{align}
The following lemmas of traces in terms of the Clifford action are
very efficient for solving the perturbed second-order spectral triples-tpye II.
\begin{lem}\label{lemaaa}The following identities hold:
\begin{align}\label{lask}
&(1){\rm tr}\left[\partial_{x_j}[c(X)]\partial_{x_j}[c(X)]\right]=-\sum_{j=1}^n|\nabla_{e_j}^{TM}X|^2{\rm tr}[id];\nonumber\\
&(2){\rm tr}\left[c(X)\partial_{x_j}[c(X)]\right]=\frac{1}{2}|X|^4\partial_{x_j}(|X|^{-2}){\rm tr}[id];\nonumber\\
&(3){\rm tr}\left[c(X)\partial_{x_j}\partial_{x_j}[c(X)]\right]=\bigg(\frac{1}{2}|X|^4\partial_{x_j}\partial_{x_j}(|X|^{-2})-|X|^6\left(\partial_{x_j}(|X|^{-2})\right)^2+\sum_{j=1}^n|\nabla_{e_j}^{TM}X|^2\bigg){\rm tr}[id];\nonumber\\
&(4){\rm tr}\left[c(X)\partial_{x_j}[c(X)]c(X)\partial_{x_j}[c(X)]\right]=\bigg(\frac{1}{2}|X|^8\left(\partial_{x_j}(|X|^{-2})\right)^2-|X|^2\sum_{j=1}^n|\nabla_{e_j}^{TM}X|^2\bigg){\rm tr}[id].
\end{align}
\end{lem}
\begin{proof}
(1) By $\nabla^{S(TM)}_{\partial_j}(x_0)=\partial_j+\sigma_j(x_0)=\partial_j,$ we get
\begin{align}
{\rm tr}\left[\partial_{x_j}[c(X)]\partial_{x_j}[c(X)]\right]&={\rm tr}\left[\nabla^{S(TM)}_{\partial_j}[c(X)]\nabla^{S(TM)}_{\partial_j}[c(X)]\right]\nonumber\\
&=\sum_{j=1}^n{\rm tr}\left[\nabla^{S(TM)}_{e_j}[c(X)]\nabla^{S(TM)}_{e_j}[c(X)]\right]\nonumber\\
&=\sum_{j=1}^n{\rm tr}\left[c(\nabla^{TM}_{e_j}X)c(\nabla^{TM}_{e_j}X)\right]\nonumber\\
&=-\sum_{j=1}^n|\nabla_{e_j}^{TM}X|^2{\rm tr}[id].
\end{align}
(2)By (\ref{kaka}) and ${\rm tr}[ab]={\rm tr}[ba]$, we have
$${\rm tr}\left[c(X)\partial_{x_j}[c(X)]\right]={\rm tr}\left[\partial_{x_j}(-|X|^{2})-\partial_{x_j}[c(X)]c(X)\right].$$
Then
\begin{align}
{\rm tr}\left[c(X)\partial_{x_j}[c(X)]\right]&={\rm tr}\left[\partial_{x_j}[c(X)]c(X)\right]\nonumber\\
&=-\frac{1}{2}\partial_{x_j}(|X|^{2}){\rm tr}[id]\nonumber\\
&=\frac{1}{2}|X|^4\partial_{x_j}(|X|^{-2}){\rm tr}[id].
 \end{align}
 (3)By (\ref{kaka}), we have
$${\rm tr}\left[c(X)\partial_{x_j}\partial_{x_j}[c(X)]\right]=\partial_{x_j}\bigg({\rm tr}\left[c(X)\partial_{x_j}[c(X)]\right]\bigg)-{\rm tr}\left[\partial_{x_j}[c(X)]\partial_{x_j}[c(X)]\right].$$
Then combining results (1) and (2), (3) holds.\\
(4)The relation of the Clifford action and ${\rm tr}[ab]={\rm tr}[ba]$ are used in the proof below.
\begin{align}
{\rm tr}\left[c(X)\partial_{x_j}[c(X)]c(X)\partial_{x_j}[c(X)]\right]&=-\partial_{x_j}(|X|^{2}){\rm tr}\left[c(X)\partial_{x_j}[c(X)]\right]+{\rm tr}\left[\partial_{x_j}[c(X)]\partial_{x_j}[c(X)]\right]\nonumber\\
&=\bigg(\frac{1}{2}\left(\partial_{x_j}(|X|^{2})\right)^2-|X|^2\sum_{j=1}^n|\nabla_{e_j}^{TM}X|^2\bigg){\rm tr}[id]\nonumber\\
&=\bigg(\frac{1}{2}|X|^8\left(\partial_{x_j}(|X|^{-2})\right)^2-|X|^2\sum_{j=1}^n|\nabla_{e_j}^{TM}X|^2\bigg){\rm tr}[id],
\end{align}
which ends the proof.
\end{proof}
By Lemma \ref{lemaaa} and (\ref{wer}), we get the following lemma.
\begin{lem}\label{ppp1}  In terms of the condition (\ref{lask}), the following identities hold:
\begin{align}\label{p1po}
&\int_{|\xi|=1}{\rm tr}\bigg(\sum_{j,l}\partial_{x_j}[c(X)]\partial_{x_l}[c(X)]\xi_j\xi_l\bigg)\sigma(\xi)=-\frac{1}{2m}\sum_{j=1}^n|\nabla_{e_j}^{TM}X|^2{\rm tr}[id]Vol(S^{n-1});\nonumber\\
&\int_{|\xi|=1}{\rm tr}\bigg(\sum_{j,l}\partial_{x_j}(|X|^{-2})c(X)\partial_{x_j}[c(X)]\bigg)\sigma(\xi)=\frac{1}{2}|X|^{4}|\nabla(|X|^{-2})|^2{\rm tr}[id]Vol(S^{n-1});\nonumber\\
&\int_{|\xi|=1}{\rm tr}\bigg(\sum_{j}c(X)\partial_{x_j}\partial_{x_j}[c(X)]\bigg)\sigma(\xi)\nonumber\\
&=\left(\sum_{j=1}^n|\nabla_{e_j}^{TM}X|^2-|X|^{6}|\nabla(|X|^{-2})|^2-\frac{1}{2}|X|^{4}\Delta(|X|^{-2})\right){\rm tr}[id]Vol(S^{n-1});\nonumber\\
&\int_{|\xi|=1}{\rm tr}\bigg(\sum_{j,l}c(X)\partial_{x_l}[c(X)]c(X)\partial_{x_j}[c(X)]\xi_j\xi_l\bigg)\sigma(\xi)\nonumber\\
&=\frac{1}{4m}\left(|X|^{8}|\nabla(|X|^{-2})|^2-2|X|^{2}\sum_{j=1}^n|\nabla_{e_j}^{TM}X|^2\right){\rm tr}[id]Vol(S^{n-1}),
\end{align}
where $\Delta(|X|^{-2})$ and $\nabla(|X|^{-2})$ denote a generalized laplacian of $|X|^{-2}$ and the gradient to $|X|^{-2}$ respectively.
\end{lem}
By integrating each item of (\ref{mnmn}) and using (\ref{666}), we establish the following theorem.
\begin{thm}\label{11thmooo}
Let $M$ be a $2m$-dimensional oriented
compact spin manifold without boundary, then the Einstein-Hilbert action for perturbed second-order spectral triples-type II is given
 \begin{align}
{\rm Wres}[(c(X)D^2c(X))^{-m+1}]
&=(-1)^m2^m\int_M\bigg(\frac{1}{12}(m-1)|X|^{-2m+2}s+\frac{1}{6}(m^2-3m+2)|X|^{-2m+4}\Delta(|X|^{-2})\nonumber\\
&+\frac{1}{12m}(6m^3-m^4+m^2-30m+24)|X|^{-2m+6}|\nabla(|X|^{-2})|^2\bigg)Vol(S^{n-1})d{\rm Vol_M}.
 \end{align}
 \end{thm}

\section*{ Declarations}
\textbf{Ethics approval and consent to participate:} Not applicable.

\textbf{Consent for publication:} Not applicable.

\textbf{Availability of data and materials:} The authors confrm that the data supporting the findings of this study are available within the article.

\textbf{Competing interests:} The authors declare no competing interests.

\textbf{Author Contributions:} All authors contributed to the study conception and design. Material preparation,
data collection and analysis were performed by TW and YW. The first draft of the manuscript was written
by TW and all authors commented on previous versions of the manuscript. All authors read and approved
the final manuscript.
\section*{Acknowledgements}
This first author was supported by NSFC. No.12401059 and Liaoning Province Science and Technology Plan Joint Project 2023-BSBA-118. The second author was supported NSFC. No.11771070. The authors thank the referee for his (or her) careful reading and helpful comments.


\section*{References}
\begin{thebibliography}{00}
\bibitem{Ac} Ackermann T. A note on the Wodzicki residue. J. Geom. Phys., 1996, 20: 404-406.

\bibitem{Co2} Connes A. The action functinal in Noncommutative geometry. Commun. Math. Phys., 1998, 117: 673-683.
\bibitem{Co1} Connes A. Quantized calculus and applications. 11th International Congress of Mathematical Physics (Paris,1994). Internat Press, Cambridge, MA. 1995, 15-36.

   \bibitem{D1}Dabrowski L, Sitarz A, Zalecki P. Spectral metric and Einstein functionals, Adv. Math., 2023, 427: 1091286.
\bibitem{D2}Dabrowski L, Sitarz A, Zalecki P. Spectral torsion, Commun. Math. Phys., 2024, 405(5): 130.
\bibitem{FM}Fries M. Relative K-homology of higher-order differential operators. J. Funct. Anal., 2025, 288(1): 110678.

\bibitem{Gu} Guillemin V W. A new proof of Weyl's formula on the asymptotic distribution of eigenvalues. Adv. Math., 1985, 55(2): 131-160.
\bibitem{KW} Kalau W, Walze M. Gravity, Noncommutative geometry and the Wodzicki residue. J. Geom. Phys., 1995, 16: 327-344.
\bibitem{Ka} Kastler D. The Dirac Operator and Gravitation. Commun. Math. Phys., 1995, 166: 633-643.

\bibitem{Si1}Sitarz A. Conformally Rescaled Noncommutative Geometries. Geometric Methods in Physics.
Trends in Mathematics. 2014, 83-100.
\bibitem{Si2}Sitarz, A. Wodzicki residue and minimal operators on a noncommutative 4-dimensional torus. J. Pseudo-Differ. Oper. Appl., 2014, 5: 305-317.
\bibitem{UW}Ugalde W J. A construction of critical GJMS operators using Wodzicki's residue. Commun. Math. Phys., 2006, 261(3): 771-788.

\bibitem{Wo} Wodzicki M. Local invariants of spectral asymmetry. Invent. Math., 1995, 75(1): 143-178.
\bibitem{Wo1} Wodzicki M. Noncommutative residue I: Fundamentals, in K-theory, Arithmetic and Geometry, Yu. I. Manin, ed., Lecture Notes in Mathematics Vol. 1289 (Springer, Berlin, 1987).


\end{thebibliography}
\end{document}